\documentclass[12pt,reqno]{amsart}
\newtheorem{theorem}{Theorem}[section]
\newtheorem{proposition}[theorem]{Proposition}
\newtheorem{lemma}[theorem]{Lemma}

\begin{document}
\baselineskip=16pt

\title[M-curves and symmetric products]{M-curves and symmetric products}

\author[I. Biswas]{Indranil Biswas}

\address{School of Mathematics, Tata Institute of Fundamental
Research, Homi Bhabha Road, Mumbai 400005, India}

\email{indranil@math.tifr.res.in}

\author[S. D'Mello]{Shane D'Mello}

\address{Department of Mathematics,
Indian Institute of Science Education and Research,
Dr. Homi Bhabha Road, Pune 411008, India}

\email{shane.dmello@iiserpune.ac.in}



\date{}

\begin{abstract}
Let $(X\, , \sigma)$ be a geometrically irreducible smooth 
projective M-curve of genus $g$ defined over the field of real numbers.
We prove that the $n$--th symmetric product of $(X\, , \sigma)$ is an M-variety
for $n\,=\,2\, ,3$ and $n \,\geq\, 2g -1$.
\end{abstract}

\maketitle

\section{Introduction}

A nonsingular geometrically irreducible real projective curve is a pair $(X\, , 
\sigma)$, where $X$ is a compact connected Riemann surface and $\sigma$ is an 
anti-holomorphic involution of $X$. The fixed point set of
this $\sigma$ is called the real 
part of $X$ and it is denoted by $X^{\sigma}$; topologically,
the subset $X^{\sigma}$ is a disjoint union of circles. 
Harnack proved that the number of components of $X^{\sigma}$ is bounded above by 
$g+1$, where $g$ is the genus of $X$~
\cite{harnack1876ueber}. Curves which have exactly the maximum number (i.e.,
genus $+1$) of components of the real
part are called M-curves. Classifying real algebraic curves up to 
homeomorphism is straightforward, however, classifying even planar non-singular real 
algebraic curves up to isotopy in $\mathbb{RP}^2$ has only been completely 
successful up to 
degree~7~\,\cite{viro1986progress,wilson1978hilbert,rokhlin1978complex}. The M-curves play 
a special role in this classification; a lot of research has been done on M-curves (see 
\cite{chislenko1982m,chevallier2002four,mishachev1975complex,korchagin1986m,shustin1987new}).
In the case of planar curves, many restrictions on the topology of M-curves are known 
(see~\cite{fiedler1983new,nikulin1984involutions}) and have helped in the rigid isotopy 
classification for low degrees~\cite{gudkov1974topology}.

The notion of an M-curve can be generalized to higher dimensional non-singular real 
algebraic varieties. For this, consider a non-singular real algebraic variety $(Y\, ,\eta)$,
where $\eta$ is an anti-holomorphic involution of a smooth
complex projective variety $Y$. As before, the fixed-point locus $Y^{\eta}$ is called
the real part of 
$(Y\, ,\eta)$. The following inequality is a generalization of Harnack's 
bound~\cite[p. 72, Corollary A.2]{wilson1978hilbert} \cite{thom1965homologie,bredon1972introduction}:
\begin{equation}\label{smith}
\sum_{i\geq 0} \dim \mathrm{H}_i(Y^{\eta},\, \mathbb{Z}/2\mathbb{Z}) \,\leq\,
\sum_{i\geq 0} \dim \mathrm{H}_i(Y,\, \mathbb{Z}/2\mathbb{Z})\, .
\end{equation}
The pair $(Y\, ,\eta)$ is called an M-variety if the inequality in \eqref{smith} 
is an equality (see \cite{BMMV}).

Since M-curves play a special role in the topology of real algebraic varieties, it is 
useful to have a criterion for M-curves. It was proved earlier 
that a curve defined over $\mathbb R$ is an M-curve if and only if its Jacobian is an 
M-variety \cite{biswas}. We use this result of \cite{biswas} and the Picard bundle to prove that the 
$n$--th symmetric product (the definition is recalled later) of an M-curve of genus $g$ is also an M-variety as long as $n 
\,\geq\, 2g-1$. Although we cannot use the Picard bundle when $n\,<\, 2g-1$, by analyzing the 
topology of the symmetric product, we are able to prove that for $n\, \leq\, 3$, the 
result still holds for all $g$. The problem remains open for $4 \,\leq\, n \,\leq\, 2g-2$.

\section{Second and third symmetric products}

Unless stated otherwise, all homologies will be considered with coefficients in
$\mathbb{Z}/2\mathbb{Z}$.

Let $X$ be a compact connected Riemann surface of genus $g$. For any $d\, \geq\, 1$, the
symmetric product $\mathrm{Sym}^d(X)$ is the quotient of $X^d$ by the natural action of
permutations of $\{1\, ,\cdots\, ,d\}$.

\begin{lemma}\label{lem1}
\mbox{}
\begin{enumerate}
\item The sum $\sum_{i\geq 0}
\dim H_i(\mathrm{Sym}^2(X);\, \mathbb{Z}/2\mathbb{Z})$ is $3+3g+2g^2$.

\item The sum $\sum_{i\geq 0}
\dim H_i(\mathrm{Sym}^3(X);\, \mathbb{Z}/2\mathbb{Z})$ is $4+14g/3+2g^2+4g^3/3$.
\end{enumerate}
\end{lemma}

\begin{proof}
The Poincar\'e polynomial (homology with coefficients in $\mathbb{Q}$) of
the $n$-th symmetric product of a topological space, whose $i$-th Betti numbers
is $B_i$, is the coefficient of $t^n$ in the power series expansion
of~\cite[p. 568, (8.5)]{macdonald1962poincare} (here $x$ is the indeterminate of the
Poincar\'e polynomial):
\begin{equation}
  \frac{(1+xt)^{B_1}(1+x^3t)^{B_3}\ldots}{(1-t)^{B_0}(1-x^2t)^{B_2}(1-x^4t)^{B_4}\ldots}
\end{equation}

For $X$ we have $B_0\,=\,B_2\,=\,1$, $B_1\,=\,2g$ and $B_i\,=\, 0$ for
all $i\, \geq\,2$. Therefore, the
Poincar\'e polynomial of $\mathrm{Sym}^n(X)$ is the coefficient of $t^n$ in 
\begin{equation}\label{el3}
  \frac{(1+xt)^{2g}}{(1-t)(1-x^2t)}\, .
\end{equation}

Setting $x\,=\,1$, the sum of the Betti numbers of $\mathrm{Sym}^k(X)$ is the coefficient
of $t^k$ in \eqref{el3}, so it is
\begin{equation}
  \begin{aligned}
    & \frac{(1+t)^{2g}}{(1-t)(1-t)}\,= \,  \frac{(1+t)^{2g}}{(1-t)^2}\,=\,
  (1+t)^{2g}(1-t)^{-2} \,=\,\\
  & (1+2gt+g(2g-1)t^2+\frac{2g(2g-2)(2g-3)}{3!}t^3+\ldots)(1+2t+3t^2+4t^3+\ldots)\, .
  \end{aligned}
\end{equation}
  Therefore, the sum of the Betti numbers of $\mathrm{Sym}^2(X)$ is
  \begin{equation}
    \label{sym2}
  3+3g+2g^2\, ,
  \end{equation}
while the sum of the Betti numbers of $\mathrm{Sym}^3(X)$ is
  \begin{equation}
    \label{sym3}
  4+14g/3+2g^2+4g^3/3\, .
  \end{equation}

The homology of $\mathrm{Sym}^n(X)$, with coefficients in $\mathbb{Z}$, is torsion-free \cite[p. 329, (12.3)]{Ma2}, so
by the universal coefficient theorem
$$\mathrm{H}_i(X;\,\mathbb{Z}/2\mathbb{Z})\,=\,
\mathrm{H}_i(X;\, \mathbb{Z}) \otimes \mathbb{Z}/2\mathbb{Z}\ \ 
\text{ and }\ \ \mathrm{H}_i(X;\mathbb{Q})\,=\,\mathrm{H}_i(X\, \mathbb{Z}) \otimes \mathbb{Q}\, .$$
Therefore, the sums of the Betti numbers
of $\mathrm{Sym}^2(X)$ and $\mathrm{Sym}^3(X)$ with $\mathbb{Z}/2\mathbb{Z}$ coefficients
are the same as in equations \eqref{sym2} and \eqref{sym3} respectively.
\end{proof}

Assume that $X$ admits an anti-holomorphic involution
$$
\sigma\, :\, X\, \longrightarrow\, X\, .
$$
For any $d\, \geq\, 1$, let $\sigma^d\, :\, \mathrm{Sym}^d(X)\, \longrightarrow\,
\mathrm{Sym}^d(X)$ be the anti-holomorphic involution defined by
$$
(x_1\, ,\cdots\, ,x_d)\, \longmapsto\,(\sigma(x_1)\, ,\cdots\, ,\sigma(x_d))\, ;
$$
in particular, $\sigma^1\,=\, \sigma$.

We will need the following lemma for the next proposition.
\begin{lemma}
The symmetric product $\mathrm{Sym}^2(\mathbb{RP}^1)$ is homeomorphic to a M\"obius band. The boundary
of $\mathrm{Sym}^2(\mathbb{RP}^1)$ is the diagonal of $\mathrm{Sym}^2(\mathbb{RP}^1)$.
\end{lemma}
\begin{proof}
Consider the subspace of 
$\mathrm{Sym}^2(\mathbb{CP}^{1})$ which is invariant under the complex conjugation.
The complement, in it, of the subspace of distinct conjugate pairs
is identified with $\mathrm{Sym}^2(\mathbb{RP}^1)$. The subspace of 
$\mathrm{Sym}^2(\mathbb{CP}^{1})$ which is invariant under the complex conjugation 
can be realized as the space of roots of real polynomials, and therefore can be 
identified as the space of real quadratic polynomials, up to scalar multiplication 
and therefore is $\mathbb{RP}^{2}$. The discriminant is a conic in this copy of 
$\mathbb{RP}^2$ and divides it into a disc and a M\"obius band. The subspace of 
conjugate pairs is naturally isomorphic to the component of $\mathbb{CP}^{1}\setminus 
\mathbb{RP}^{1}$, which is a disc. Therefore, $\mathrm{Sym}^2(\mathbb{RP}^{1})$ is 
the projective plane minus a disc which is a M\"obius band. Note that the boundary 
consists of the discriminant conic which corresponds to the space of repeated roots, 
which is the diagonal of $\mathrm{Sym}^2(\mathbb{RP}^1)$.
\end{proof}

\begin{proposition}\label{pla}
If $(X\, ,\sigma)$ is an M-curve, then $(\mathrm{Sym}^2(X)\, , \sigma^2)$ is an M-variety.
\end{proposition}

\begin{proof}
Let $X^\sigma\, \subset\, X$ be the subset fixed by $\sigma$.
Assume that $(X\, , \sigma)$ is an M-curve, so $X^\sigma$ is a disjoint union
$\sqcup_{i=1}^{g+1} C_i$ of $g+1$ circles. Note that the Euler characteristic of $\chi (X\setminus X^{\sigma}) = \chi(X) = 2-2g$ because $X^\sigma$ is the union of the
circles $C_i$. Suppose $X\setminus X^\sigma$ is connected, note that its closure has $2(g+1)$ boundary components. After gluing a disc to each boundary component of $X\setminus X^\sigma$, one obtains a connected surface whose Euler characteristic is $2-2g + 2(1+g)=4$ which is impossible for a closed surface. Therefore, 
the complement $X\setminus X^{\sigma}$ has two connected components $X_1$ and $X_2$ that
are interchanged by $\sigma$, and $X^{\sigma}$ is the boundary of both $\overline{X_1}$
and $\overline{X_2}$. 

The real part of $\mathrm{Sym}^2(X)$ is the disjoint union of the following:
  \begin{enumerate}
\item $\overline{\{(z\, , \overline{z})\,\mid\, z\in X_1\}}\cup \{(\lambda\, ,\mu)\,
\mid\, \lambda\, , \mu\,\in\, C_i\}$, glued along $\{(\lambda\, ,\lambda)\,\mid\, \lambda\in C_i\}$. Using the previous lemma, this is
homeomorphic to the space $X_1$ with a M\"obius band attached along each component of the
boundary. Note that $\mathrm{Sym}^2(C_i)$ is a M\"obius band because $C_i$ is homeomorphic to $\mathbb{RP}^1$. 
\item $\{(\lambda\, ,\mu) \,\mid\, \lambda \in C_i, \mu \in C_j, i\neq j\}$;
it is a disjoint union of ${g+1}\choose{2}$ tori $S^1\times S^1$.
\end{enumerate}

Denote by $Y$ the space $\overline{X_1}$ with a M\"obius band glued at each
component of  its boundary. The fixed point locus $\mathrm{Sym}^2(X)^{\sigma^2}$ is
isomorphic to the disjoint union of $Y$ 
and ${g+1}\choose{2}$ tori. The sum of the Betti numbers of ${g+1}\choose{2}$ tori is $4\cdot
{{g+1}\choose{2}}\,=\,2g(g+1)$.

Note that the $0$-th and $2$-nd Betti numbers of $Y$ are 1 since we are considering
homology with $\mathbb{Z}/2\mathbb{Z}$ coefficients. Therefore, the sum of Betti numbers
of $Y$ is $4-\chi(Y)$. The Euler characteristic $\chi(Y)\,=\,\chi(X)/2\,=\,1-g$ because
$X^\sigma$, being a union of circles, has Euler characteristic 0 and the Euler
characteristic of the two halves $X_1$ and $X_2$ are the same; the Euler characteristic of a
M\"obius band is 0. Therefore, $4-\chi(Y)\,=\,4-(1-g)\,=\,3+g$. Therefore, the
sum of the Betti numbers of $\mathrm{Sym}^2(X)^{\sigma^2}$ is $2g(g+1)+3+g\,=\,3+3g+2g^2$.
Now the proposition follows from Lemma \ref{lem1}(1).
\end{proof}

In the following theorem, we treat $S^1$ as the subspace of complex numbers with unit 
norm. It has a group structure under complex multiplication.

\begin{theorem}
  \label{morton}
\begin{enumerate}
  \item $\theta : \mathrm{Sym}^3(S^1) \,\longrightarrow\, S^1$, where $\theta(\lambda_1, \lambda_2, \lambda_3)= \lambda_1\lambda_2\lambda_3$, is a fiber bundle over $S^1$ with fiber homeomorphic to a two simplex.
  \item  The fiber $\theta^{-1}(\mu)$ has boundary $\{(\lambda_1, \lambda_1, \lambda_2) \ |\ \lambda_1, \lambda_2 \in S^1,\  \lambda_1\lambda_1\lambda_2 = \mu\} \cong S^1 $. The boundary of  $\mathrm{Sym}^3(S^1)$ is the space $\{(\lambda_1, \lambda_1, \lambda_2) \ |\ \lambda_1, \lambda_2 \in S^1\}\cong S^1 \times S^1$.
\end{enumerate}
\end{theorem}
The proof of (1) is a special case of the proof in \cite{morton}, and is based
on the version of the proof described
in \cite{St}. We recall the proof in order to describe the boundary.

\begin{proof}
Any point $\mu$ of $S^1$ is contained in a locally trivial neighborhood: For the 
neighborhood $U:= \{\mu e^{2 \pi\sqrt{-1} s} \ |\ s \in (-\frac{1}{2}, \frac{1}{2}) \}\ni \mu$, the map $\psi: 
U \times \theta^{-1}(\mu) \,\longrightarrow\, \theta^{-1}(U)$ given by $\psi(\mu e^{2\pi\sqrt{-1} s}, 
(\lambda_1, \lambda_2, \lambda_3))= (\lambda_1e^{2\pi \sqrt{-1} s/3}, \lambda_2e^{2\pi\sqrt{-1} 
s/3}, \lambda_3e^{2\pi\sqrt{-1} s/3})$ is clearly a local trivialization.

  We will now show that the fiber $\theta^{-1}(1)$ is homeomorphic to  the standard 2-simplex
$$\Delta^2\,:=\,\{(d_1, d_2)\ |\ d_i \in [0, 1], d_1+d_2 \leq 1\}$$ via the map
$$t\,:\,\Delta^2 \,\longrightarrow\, \theta^{-1}(1) \,\subset\, \mathrm{Sym}^3(S^1)$$ defined by $t(d_1, d_2) := (\lambda, \lambda e^{2\pi\sqrt{-1} d_1}, \lambda e^{2\pi\sqrt{-1} (d_1+d_2)})$, where $\lambda := e^{-2\pi\sqrt{-1}(\frac{2d_1+d_2}{3})}\in S^1$ (note that $\lambda$ was chosen so that $\lambda \lambda e^{2\pi\sqrt{-1} d_1} \lambda e^{2\pi\sqrt{-1} (d_1+d_2)}=1$, to ensure that the image of $t$ lies in $\theta^{-1}(1)$). Explicitly, $t(d_1, d_2) := (e^{-2\pi\sqrt{-1}(\frac{2d_1+d_2}{3})}, e^{2\pi\sqrt{-1} (d_1-\frac{2d_1+d_2}{3})}, e^{2\pi\sqrt{-1} (d_1+d_2-\frac{2d_1+d_2}{3})})$; observe that the powers of $e$ sum up to 0.
  
  Given $(\lambda_1, \lambda_2, \lambda_3) \in \mathrm{Sym}^3(S^1)$, while there are several triples $(s_1, s_2, s_3) \in \mathbb{R}^3$, such that $s_1\leq s_2 \leq s_3 \leq s_1+1$, $e^{2\pi\sqrt{-1} s_1}=\lambda_i$, $e^{2\pi\sqrt{-1} s_2}=\lambda_j$, and $e^{2\pi\sqrt{-1} s_3}=\lambda_k$ where $i \neq j\neq k$, only one triple also satisfies $s_1+s_2+s_3=0$, which is a necessary condition to define the inverse, owing to the last sentence in the previous paragraph. We will prove that this condition is also sufficient to ensure that the map $t'((\lambda_1, \lambda_2, \lambda_3)) = (s_2-s_1, s_3-s_2)$ is the inverse of $t$.

  $\mathbb{R}$ covers the circle via the map $s \,\longmapsto\, e^{2\pi\sqrt{-1} s}$ so we can pick $(s_1, s_2, s_3) \in \mathbb{R}^3$, such that $s_1\leq s_2 \leq s_3 \leq s_1+1$, $e^{2\pi\sqrt{-1} s_1}=\lambda_1$, $e^{2\pi\sqrt{-1} s_2}=\lambda_2$, and $e^{2\pi\sqrt{-1} s_3}=\lambda_3$ by lifting along the path that begins at $\lambda_1$ and wraps around the circle once.  Since $\lambda_1\lambda_2\lambda_3 = 1$, $s_1+s_2+s_3 \in \mathbb{Z}$. Define $T(s_1, s_2, s_3) = (s_2, s_3, s_1+1)$, then $(s_2', s_3', s_1')=T(s_1, s_2, s_3)$ also satisfies $s_2'\leq s_3' \leq s_1' \leq s_1'+1$, $e^{2\pi\sqrt{-1} s_1'}=\lambda_1$, $e^{2\pi\sqrt{-1} s_2'}=\lambda_2$, and $e^{2\pi\sqrt{-1} s_3'}=\lambda_3$, however $s_1'+s_2'+s_3'=(s_1+s_2+s_3)+1$. 
  
  $\{T^n(s_1, s_2, s_3)\ |\ n \in \mathbb{Z}\}$ consists of all possible triples 
$(s_i', s_j', s_k')$ that satisfy $s_i'\leq s_j' \leq s_k' \leq s_i'+1$, $e^{2\pi\sqrt{-1}
s_i'}=\lambda_i$, $e^{2\pi\sqrt{-1} s_j'}=\lambda_j$, and $e^{2\pi\sqrt{-1} s_k'}=\lambda_k$, $i 
\neq j \neq k\neq i$, but the sums $s_i' + s_j' + s_k'$ are all distinct
and they span all 
of $\mathbb{Z}$. Therefore there exists a unique $(s_1', s_2', s_3') \in \{T^n(s_1, 
s_2, s_3)\ |\ n \in \mathbb{Z}\}$ that sum to $0$ and the map $t'((\lambda_1, 
\lambda_2, \lambda_3)) = (s_2'-s_1', s_3'-s_2')$ is well-defined.
  
  It is easy to check that $t'$ is an inverse of $t$. Indeed, $t \circ t'(\lambda_1, \lambda_2, \lambda_3)= t(s_2'-s_1', s_3'-s_2') = (e^{2\pi\sqrt{-1} s_1'}, e^{2\pi\sqrt{-1} s_2'}, e^{2\pi\sqrt{-1} s_3'})$, where the last equality also follows from the fact that $s_1'+s_2'+s_3'=0$. On the other hand, since $s_1'=-\frac{2d_1+d_2}{3}$, $s_2'=d_1-\frac{2d_1+d_2}{3}$, and $s_3'=d_1+d_2-\frac{2d_1+d_2}{3}$ already sums to 0, $t'(e^{-2\pi\sqrt{-1}(\frac{2d_1+d_2}{3})}, e^{2\pi\sqrt{-1} (d_1-\frac{2d_1+d_2}{3})}, e^{2\pi\sqrt{-1} (d_1+d_2-\frac{2d_1+d_2}{3})}) = (d_1, d_2)$ and therefore, $t' \circ t = \mathrm{Id}$.

  For (2), it follows from the definition of $t$ that given $(\lambda_1, \lambda_2, \lambda_3) = t(d_1, d_2)$, $\lambda_i = \lambda_j$ if and only if $d_1=0$, $d_2=0$, or $d_1+d_2=1$.  This boundary $\{(\lambda_1, \lambda_1, \lambda_2) \ |\ \lambda_1, \lambda_2 \in S^1\}\subset \mathrm{Sym}^3(S^1)$ is homeomorphic to $S^1 \times S^1$ via
the homeomorphism given by $(\lambda_1, \lambda_1, \lambda_2) \,\longmapsto\, (\lambda_1, \lambda_2)$, where the $\lambda_1$ is the repeated element. 
\end{proof}

\begin{lemma}
  \label{components}
Let $X_1$ be one half of $X\setminus X^{\sigma}$. 
Then $\mathrm{Sym}^3(X)^{\sigma^3}$ is the disjoint union of ${g+1}\choose{3}$ copies of the set $S^1 \times S^1 \times S^1$ and the connected sets $B_i$, $i=1 \ldots g+1$, where each $B_i$ is constructed as follows: $S^1 \times X_1$ has $g+1$ disjoint copies of $S^1 \times S^1$ in the boundary; to the $i$th boundary, glue $\mathrm{Sym}^3(S^1)$ (which also has boundary $S^1 \times S^1$) and to each of the other $g$ boundaries glue $S^1 \times M$ (which also has $S^1 \times S^1$ as its boundary) where $M$ is a M\"obius band (the precise gluings will be described in the proof). Furthermore, $\mathrm{H}_1(B_i;\mathbb{Z}/2\mathbb{Z}) \cong H_1(S^1 \times X_1;\mathbb{Z}/2\mathbb{Z})$ for each $i$.
\end{lemma}

\begin{proof}
Denote the $i$th boundary component of $X_1$ by $C_i$. The space 
$\mathrm{Sym}^3(X)^{\sigma^3}$ consists of elements of the form $\{ (\lambda_1\, , 
\lambda_2\, , \lambda_3) \,\mid\, \lambda_i \in X^\sigma\}$ or $\{ (\lambda\, , z\, , 
\bar{z}) \,\mid\, \lambda \in X^\sigma, z \in X \setminus X^\sigma \}$. The latter is 
clearly homeomorphic to the disjoint union of $Y_0^i := \{ (\lambda\, , z\, , 
\bar{z}) \,\mid\, \lambda \in C_i, z \in X \setminus X^\sigma \} = C_i \times X_1 
\cong S^1 \times X_1$ (see Proposition \ref{pla}). The former consists of spaces of three forms:
\begin{enumerate}
    \item $\{ (\lambda_1\, , \lambda_2\, , \lambda_3) \,\mid\, \lambda_1 \in C_i, \lambda_2 \in C_j, \lambda_3 \in C_k, i \neq j \neq k\}$. This space is $C_i \times C_j \times C_k \cong S^1 \times S^1 \times S^1$. There are ${g+1}\choose{3}$ copies of it and they are clearly disjoint.
  \item $\{ (\lambda_1\, , \lambda_2\, , \lambda_3) \,\mid\, \lambda_1 \in C_i, \lambda_2, \lambda_3 \in C_j, i \neq j\}$. This space is $C_i \times \mathrm{Sym}^2(C_j) \cong S^1 \times M$, where $M$ is the M\"obius band. Note that the boundary of this space is $\{ (\lambda_1\, , \lambda_2\, , \lambda_2) \,\mid\, \lambda_1 \in C_i, \lambda_2 \in C_j, i \neq j\}$ and it lies in the closure and therefore the same
component of $C_i\times X_1$. 
  \item $\{ (\lambda_1\, , \lambda_2\, , \lambda_3) \,\mid\, \lambda_1, \lambda_2, \lambda_3 \in C_i\}$. This space is $\mathrm{Sym}^3(C_i) \cong \mathrm{Sym}^3(S^1)$. Its boundary is $\{ (\lambda_1\, , \lambda_2\, , \lambda_2) \,\mid\, \lambda_1, \lambda_2 \in C_i \}$, which is also in the closure  and therefore the same component of $C_i \times X_1$.
  \end{enumerate}
  We claim that $$\mathrm{H}_1(Y_0^i \cup C_i \times \mathrm{Sym}^2(C_j);
\mathbb{Z}/2\mathbb{Z}) \cong \mathrm{H}_1(\overline{Y_0^i};\mathbb{Z}/2\mathbb{Z})$$
for all $i\,\not=\, j$. Note that $\overline{Y_0^i} \cap C_i \times \mathrm{Sym}^2(C_j) = \{ (\lambda_1\, , \lambda_2\, , \lambda_2) \,\mid\, \lambda_1 \in C_i, \lambda_2 \in C_j, i \neq j\} \cong S^1 \times S^1$.  Consider the reduced homology version of the Mayer Vietoris sequence,
 \[ \ldots \,\longrightarrow\, \widetilde{\mathrm{H}}_1(S^1 \times S^1) \,\longrightarrow\,
\widetilde{\mathrm{H}}_1(\overline{Y_0^i}) \oplus
\widetilde{\mathrm{H}}_1(C_i \times \mathrm{Sym}^2(C_j))\,\longrightarrow\,
\widetilde{\mathrm{H}}_1(\overline{Y_0^i} \cup C_i \times \mathrm{Sym}^2(C_j))\,\longrightarrow\, 0\]
which simplifies to,
\[ \ldots \,\longrightarrow\, \mathbb{Z}/2\mathbb{Z} \oplus \mathbb{Z}/2\mathbb{Z} \xrightarrow{\theta}
\widetilde{\mathrm{H}}_1(S^1 \times X_1) \oplus (\mathbb{Z}/2\mathbb{Z} \oplus
\mathbb{Z}/2\mathbb{Z}) \,\longrightarrow\, \widetilde{\mathrm{H}}_1(Z) \,\longrightarrow\, 0\]
where $Z\,=\, \overline{Y_0^i} \cup C_i \times \mathrm{Sym}^2(C_j)$.

Let $[A_1]$ and $[A_2]$ be the generators of $\widetilde{\mathrm{H}}_1(S^1 \times S^1)= 
\mathbb{Z}/2\mathbb{Z} \oplus \mathbb{Z}/2\mathbb{Z}$, represented by the curves $A_1=\{ (\lambda_1\, , 
1\, , 1) \,\mid\, \lambda_1 \in C_i, 1 \in C_j\}$ and $A_2=\{ (1\, , 
\lambda_2\, , \lambda_2) \,\mid\, 1 \in C_i, \lambda_2 \in C_j\}$. Note 
that $A_1$ and $A_2$ are not boundaries in $S^1 \times X_1$ because $X_1$ has more 
than one boundary component, however, $A_2$ is twice a generator in $S^1 \times 
\mathrm{Sym}^2(C_j)\cong S^1\times M$, where $M$ is the M\"obius band, and is 
therefore $0$ (since the coefficient is $\mathbb{Z}/2\mathbb{Z}$). Therefore,

\[[A_1] \xrightarrow{\theta} ([A_1], [A_1])\]
\[[A_2] \xrightarrow{\theta} ([A_2], 0)\]

Therefore, the image of $\theta$ is $\mathbb{Z}/2\mathbb{Z} \oplus \mathbb{Z}/2\mathbb{Z}$, from which it follows that $\mathrm{H}_1(\overline{Y_0^i} \cup C_i \times \mathrm{Sym}^2(C_j);\mathbb{Z}/2\mathbb{Z}) \cong \mathrm{H}_1(\overline{Y_0^i};\mathbb{Z}/2\mathbb{Z})$. This proves the claim.

Following exactly the same proof, one can prove that taking the union of $Y_0^i$
with all spaces
  $C_i \times \mathrm{Sym}^2(C_j)$, $j\, \not=\, i$, does not change
the first homology. Denote the resulting space by $Y^i$. The boundary of $Y^i$ is
$C_i\times C_i$.

We now claim that
$$\mathrm{H}_1(\overline{Y^i} \cup \mathrm{Sym}^3(C_i);\mathbb{Z}/2\mathbb{Z})
\cong \mathrm{H}_1(\overline{Y^i};\mathbb{Z}/2\mathbb{Z})\, .$$ Note that $\overline{Y^i} \cap \mathrm{Sym}^3(C_i) = \{ (\lambda_1\, , \lambda_2\, , \lambda_2) \,\mid\, \lambda_1, \lambda_2 \in C_i\} \cong S^1 \times S^1$. Consider the Mayer Vietoris sequence,
\[ \ldots \,\longrightarrow\, \widetilde{\mathrm{H}}_1(S^1 \times S^1) \,\longrightarrow\, \widetilde{\mathrm{H}}_1(\overline{Y^i}) \oplus \widetilde{\mathrm{H}}_1(\mathrm{Sym}^3(C_i)) \,\longrightarrow\, \widetilde{\mathrm{H}}_1(\overline{Y^i} \cup \mathrm{Sym}^3(C_i)) \,\longrightarrow\, 0\]
which simplifies to
\[ \ldots \,\longrightarrow\, \mathbb{Z}/2\mathbb{Z} \oplus \mathbb{Z}/2\mathbb{Z} \xrightarrow{\theta} \widetilde{\mathrm{H}}_1(\overline{Y^i}) \oplus \mathbb{Z}/2\mathbb{Z} \,\longrightarrow\, \widetilde{\mathrm{H}}_1(\overline{Y^i} \cup D^2 \times S^1) \,\longrightarrow\, 0\]

Let $[A_1]$ and $[A_2]$ be the generators of $\widetilde{\mathrm{H}}_1(S^1 \times S^1)= \mathbb{Z}/2\mathbb{Z} \oplus \mathbb{Z}/2\mathbb{Z}$, represented by $A_1 := \{(\lambda, \lambda, 1)\ |\ \lambda \in S^1\}$ and $A_2 := \{(1, 1, \lambda) \ |\ \lambda \in S^1\}$. Let us consider the curves $A_1':=\{(\lambda, \lambda, \mu) \ |\ \lambda, \mu \in C_i,\ \lambda^2\mu=1\}$ and $A_2' := \{(1, 1, \lambda) \ |\ \lambda \in S^1\}$ in $\mathrm{Sym}^3(C_i)$. By theorem~\ref{morton}, $A_1'$ is the boundary of a fiber and therefore homologous to zero, whereas $A_2'$ is a section of the fiber bundle $\theta: \mathrm{Sym}^3(S^1) \,\longrightarrow\, S^1$, and therefore $[A_2']$ generates the first homology of $\mathrm{Sym}^3(C_i)$. Note that $[A_1']$ and $[A_2']$ are generators of the first homology of the boundary of $\mathrm{Sym}^3(C_i)$. Let the image of $A_1$ in the boundary of $\mathrm{Sym}^3(S^1)$ represent the homology class $a[A_1']+b[A_2']$. The coefficients can be determined by the mod 2 intersection number with $A_1'$ and $A_2'$: $A_1$ intersects $A_2'$ in one point, namely $(1,1,1)$, therefore $a=1$; $A_1$ intersects $A_1'$ in two points, namely $\{(\lambda, \lambda, 1)\ |\ \lambda^2=1\}= \{(1,1,1), (-1, -1, 1)\}$, therefore the intersection number mod 2 is 0 and $b=0$. Therefore, the image of $A_1$ in the boundary of $\mathrm{Sym}^3(S^1)$ represents the homology class $[A_1']$ which in the solid torus is 0. On the other hand, $A_2=A_2'$, by definition. Now the image of $A_1$ in $\overline{Y^i}$ is the boundary and therefore homologous to 0, whereas the image of $A_2$ in $\overline{Y^i}$ represents a generator of the first homology. Therefore,

\[[A_1] \xrightarrow{\theta} (0, 0)\]
\[[A_2] \xrightarrow{\theta} ([A_2], [A_2'])\]

Therefore, the image of $\theta$ is $\mathbb{Z}/2\mathbb{Z}$ and $\widetilde{\mathrm{H}}_1(\overline{Y^i}) \cong \widetilde{\mathrm{H}}_1(\overline{Y^i} \cup \mathrm{Sym}^3(C_i))$.
\end{proof}

\begin{lemma}
  The sum of the Betti numbers of each $B_i$ is $2(g+2)$.
\end{lemma}

\begin{proof}
 We already know, by computing the Euler characteristic, that $\mathrm{H}_1(X_1;
\mathbb{Z}/2\mathbb{Z}) \,=\, (\mathbb{Z}/2\mathbb{Z})^g$. By the K\"unneth formula, we
have that $\mathrm{H}_1(S^1 \times X_1; \mathbb{Z}/2\mathbb{Z}) \,=\,
(\mathbb{Z}/2\mathbb{Z})^{g+1}$. By the previous lemma, $\mathrm{H}_1(B_i;\mathbb{Z}/2\mathbb{Z})
\,\cong \,\mathrm{H}_1(S^1 \times X_1;\mathbb{Z}/2\mathbb{Z})\,=\,
(\mathbb{Z}/2\mathbb{Z})^{g+1}$. Finally, by the Poincar\'e duality coupled with the
universal coefficient theorem
of cohomology, $\mathrm{H}_2(B_i;\mathbb{Z}/2\mathbb{Z}) \,=\,
\mathrm{H}_1(B_i;\mathbb{Z}/2\mathbb{Z})$. Since $B_i$ is a compact manifold without
boundary, the sum of the Betti numbers is $1 + (g+1) + (g+1) + 1 \,=\, 2(g+2)$
\end{proof}

\begin{proposition}
Let $(X\, ,\sigma)$ be an M-curve. Then $(\mathrm{Sym}^3(X)\, , \sigma^3)$ is an M-variety.
\end{proposition}

\begin{proof}
  By lemma~\ref{components}, $\mathrm{Sym}^3(X)^{\sigma^3}$ has ${g+1}\choose{3}$ components, disjoint copies of $S^1 \times S^1 \times S^1$, each of which has the sum of Betti numbers 8. Since the sum of the Betti numbers of each $B_i$ is $2(g+2)$ and there are $g+1$ of them, the total sum of the Betti numbers is $8 {{g+1}\choose{3}} + (g+1)2(g+2)$.

Now the proposition follows from Lemma \ref{lem1}(1).
\end{proof}

\section{Picard bundles and higher symmetric powers}

\begin{theorem}
Let $(X\,,\sigma)$ be an M-curve of genus $g$. For any $n \,\geq\, 2g-1$, the
symmetric product $(\mathrm{Sym}^n(X)\, ,\sigma^n)$ is also an M-variety.
\end{theorem}

\begin{proof}
Note that if the real part of $(X\,,\sigma)$ is empty, it cannot be an M-curve. Therefore, 
in what follows, we will assume that the real part of $(X\,,\sigma)$ is non-empty.

Fix a point $p\, \in\, X^\sigma$. Take any $n\,\geq\, 2g-1$.
Consider the map
\begin{equation}\label{u}
u\,:\,\mathrm{Sym}^n(X) \,\longrightarrow \,\mathrm{Pic}^0(X)
\end{equation}
defined by $u((x_1\, 
,x_2\, ,\ldots\, ,x_n))\,=\, \mathcal{O}_X(\sum_{i=1}^n (x_i - p))$. The map $u$ is
surjective. Indeed, for any $L\,\in\, \mathrm{Pic}^0(X)$, the line bundle
$L\otimes {\mathcal O}_X(np)$ admits nonzero holomorphic sections because by
Riemann--Roch,
$$
\dim H^0(X,\, L\otimes {\mathcal O}_X(np)) - \dim H^1(X,\, L\otimes {\mathcal O}_X(np))
\,=\, n -g +1\, >\, 0\, .
$$
Therefore, if $s$ is a nonzero holomorphic section of $L\otimes {\mathcal O}_X(np)$, then
we have
$$
L\, =\, u(\text{divisor}(s))\, .
$$
Hence $u$ is surjective. The fiber of $u$ over $L$ is the space of effective divisors on
$X$ of degree $n$ linearly equivalent to $L\otimes {\mathcal O}_X(np)$. Using
Serre duality,
$$
H^1(X,\, L\otimes {\mathcal O}_X(np))\,=\, H^0(X,\, L^*\otimes {\mathcal O}_X(-np)\otimes
K_X)^*\, =\, 0
$$
because
$$
\text{degree}(L^*\otimes {\mathcal O}_X(-np)\otimes K_X)\,=\, -n+2g-2\, <\, 0\, .
$$
Hence by Riemann--Roch, we have
$$
\dim H^0(X,\, L\otimes {\mathcal O}_X(np))\,=\, n-g+1\, .
$$
Therefore, $u$ makes $\mathrm{Sym}^n(X)$ a projective bundle over $\mathrm{Pic}^0(X)$
of relative dimension $n-g$; see \cite[Ch.~IV, \S~3]{arbarello1985geometry} for more
details.

Let
$$
\sigma^*\, :\, \mathrm{Pic}^0(X)\,\longrightarrow\, \mathrm{Pic}^0(X)\, , \ \
\mathcal{O}_X(D)\, \longmapsto\, \mathcal{O}_X(\sigma(D))
$$
be the anti-holomorphic involution given by $\sigma$. Note that for a meromorphic function
$f$ on $X$, we have $\sigma(\text{divisor}(f))\,=\, \text{divisor}
(\overline{f\circ\sigma})$. Therefore, $\sigma^*$ is well-defined, and
$\sigma^n$ takes fibers of the map $u$ in \eqref{u} to fibers of $u$. We now prove that
the restriction of $u$ to $\mathrm{Sym}^n(X)^{\sigma^n}$ surjects onto 
$\mathrm{Pic}^0(X)^{\sigma^*}$.

Take any $\mathcal{O}_X(D)\,\in\, 
\mathrm{Pic}^0(X)^{\sigma^*}$. Since the real part of the curve is non-empty, the line 
bundle $\mathcal{O}_X(D)$ cannot have a quaternionic 
structure~\cite[p. 206, Propositon 3.1 and p. 210, \S~4.2]{biswashuismanhurtubise},
\cite{bhoslebiswas} (there are no quaternionic bundles of 
odd rank on $X$ because $X^\sigma$ is nonempty). Therefore, the line bundle
$\mathcal{O}_X(D)$ has a real structure, i.e., there is a lift 
$\widetilde{\sigma}$ of $\sigma$ to $\mathcal{O}_X(D)$ such that $\widetilde{\sigma} \circ 
\widetilde{\sigma}$ is the identity map of $\mathcal{O}_X(D)$, and $\widetilde{\sigma}$
is fiberwise conjugate linear. For any holomorphic section $s$ of $\mathcal{O}_X(D+
n.p)$, the section $\varphi(s) \,:= \,s+\widetilde{\sigma}(s)$ is clearly fixed by
$\widetilde{\sigma}$.

We will show that there is a holomorphic section $s$ such that $\varphi(s)\, \not=\, 0$.
For this, first note that the homomorphism
$$
H^0(X,\, \mathcal{O}_X(D+n.p))\, \longrightarrow\,
H^0(X,\, \mathcal{O}_X(D+n.p))\, , \ \ s\, \longmapsto\,
\widetilde{\sigma}(s)
$$
is conjugate linear. Hence if $s'$ is a nonzero holomorphic section such that
$\varphi(s')\,=\, s'+\widetilde{\sigma}(s')\,=\, 0$, then
$$
\varphi(\sqrt{-1}s')\,=\, \sqrt{-1}s'+\widetilde{\sigma}(\sqrt{-1}s')\,=\,
\sqrt{-1}s'-\sqrt{-1}\widetilde{\sigma}(s')\,=\, 2\sqrt{-1}s'\, .
$$
So $\varphi(\sqrt{-1}s')\, \not=\, 0$ because $s'\, \not=\, 0$.

Take any $s\,\in\, H^0(X,\, \mathcal{O}_X(D+n.p))$ such that $\varphi(s)\, \not=\, 0$.
Therefore, $D_s\, :=\, \text{divisor}(s)$ is an element of $\mathrm{Sym}^n(X)^{\sigma^n}$
such that $u(D_s)\,=\, \mathcal{O}_X(D)\,\in\, \mathrm{Pic}^0(X)^{\sigma^*}$. Hence
\begin{equation}\label{ur}
u\vert_{\mathrm{Sym}^n(X)^{\sigma^n}}\, :\, \mathrm{Sym}^n(X)^{\sigma^n}\,\longrightarrow
\, \mathrm{Pic}^0(X)^{\sigma^*}
\end{equation}
is surjective.

Consequently, the map in \eqref{ur} makes $\mathrm{Sym}^n(X)^{\sigma^n}$ a real projective
bundle over $\mathrm{Pic}^0(X)^{\sigma^*}$ of fixed relative dimension $n-g$.

If $X$ is an M-curve, then $\mathrm{Pic}^0(X)$ is an M-variety \cite[Proposition 2.3]{biswas}. 
The complex projective space with the standard anti-holomorphic
involution is also an M-variety. Now it follows from the Leray--Hirsch 
theorem that $\mathrm{Sym}^n(X)$ is an M-variety.
\end{proof}

\end{document}